\documentclass[reqno]{amsart}
\usepackage[T1]{fontenc}
\usepackage{amsmath}
\usepackage{amssymb}
\usepackage{amsthm}
\begin{document}
\def\eq#1{{\rm(\ref{#1})}}
\theoremstyle{plain}
\newtheorem*{theo}{Theorem}
\newtheorem*{ack}{Acknowledgements}
\newtheorem*{pro}{Proposition}
\newtheorem*{coro}{Corollary}
\newtheorem*{lemm}{Lemma}
\newtheorem{thm}{Theorem}[section]
\newtheorem{lem}[thm]{Lemma}
\newtheorem{prop}[thm]{Proposition}
\newtheorem{cor}[thm]{Corollary}
\theoremstyle{definition}
\newtheorem{dfn}[thm]{Definition}
\newtheorem*{rem}{Remark}
\def\coker{\mathop{\rm coker}}
\def\ind{\mathop{\rm ind}}
\def\Re{\mathop{\rm Re}}
\def\Vol{\rm Vol}
\def\Im{\mathop{\rm Im}}
\def\im{\mathop{\rm im}}
\def\sp{\mathop{\rm span}}
\def\Hol{{\textstyle\mathop{\rm Hol}}}
\def\C{{\mathbin{\mathbb C}}}
\def\R{{\mathbin{\mathbb R}}}
\def\N{{\mathbin{\mathbb N}}}
\def\Z{{\mathbin{\mathbb Z}}}
\def\O{{\mathbin{\mathbb O}}}
\def\L{{\mathbin{\mathcal L}}}
\def\X{{\mathbin{\mathcal X}}}
\def\al{\alpha}
\def\be{\beta}
\def\ga{\gamma}
\def\de{\delta}
\def\ep{\epsilon}
\def\io{\iota}
\def\ka{\kappa}
\def\la{\lambda}
\def\ze{\zeta}
\def\th{\theta}
\def\vt{\vartheta}
\def\vp{\varphi}
\def\si{\sigma}
\def\up{\upsilon}
\def\om{\omega}
\def\De{\Delta}
\def\Ga{\Gamma}
\def\Th{\Theta}
\def\La{\Lambda}
\def\Om{\Omega}
\def\Up{\Upsilon}
\def\sm{\setminus}
\def\na{\nabla}
\def\pd{\partial}
\def\op{\oplus}
\def\ot{\otimes}
\def\bigop{\bigoplus}
\def\iy{\infty}
\def\ra{\rightarrow}
\def\longra{\longrightarrow}
\def\dashra{\dashrightarrow}
\def\t{\times}
\def\w{\wedge}
\def\bigw{\bigwedge}
\def\d{{\rm d}}
\def\bs{\boldsymbol}
\def\ci{\circ}
\def\ti{\tilde}
\def\ov{\overline}
\def\sv{\star\vp}
\title[Hamiltonian Structures in $G_2$-Geometry]{Remarks on Hamiltonian Structures in $G_2$-Geometry}

\author[Cho, Salur, and Todd]{Hyunjoo Cho, Sema Salur, and A. J. Todd}

\address {Department of Mathematics, University of Rochester, Rochester, NY, $14627$}
\email{cho@math.rochester.edu}

\address {Department of Mathematics, University of Rochester, Rochester, NY, $14627$}
\email{salur@math.rochester.edu}

\address {Department of Mathematics, University of California - Riverside, Riverside, CA, $92521$}
\email{ajtodd@math.ucr.edu}

\begin{abstract}
In this article, we treat $G_2$-geometry as a special case of multisymplectic geometry and make a number of remarks regarding Hamiltonian multivector fields and Hamiltonian differential forms on manifolds with an integrable $G_2$-structure; in particular, we discuss existence and make a number of identifications of the spaces of Hamiltonian structures associated to the two multisymplectic structures associated to an integrable $G_2$-structure. Along the way, we prove some results in multisymplectic geometry that are generalizations of results from symplectic geometry.
\end{abstract}

\date{}
\maketitle
\section{Introduction}
Let $M$ be a $7$-dimensional manifold admitting a smooth differential $3$-form $\vp$ such that, for all $p\in M$, the pair $(T_pM,\vp)$ is isomorphic as an oriented vector space to the pair $(\R^7,\vp_0)$ where 
\begin{equation}
\vp_0=\d x^{123}+\d x^{145}+\d x^{167}+\d x^{246}-\d x^{257}-\d x^{347}-\d x^{356}
\label{G23form}
\end{equation}
with $\d x^{ijk}=\d x^i\w \d x^j\w \d x^k$. In \cite{Br1}, it is shown that the Lie group $G_2$ can be defined as the set of all elements of $GL(7,\R)$ that preserve $\vp_0$, so for a manifold admitting such a $3$-form, there is a reduction in the structure group of the tangent bundle to the exceptional Lie group $G_2$; hence, the pair $(M,\vp)$ is called a \emph{manifold with $G_2$-structure}. Using the theory of $G$-structures and the inclusion of $G_2$ in $SO(7)$, all manifolds with $G_2$-structure are necessarily orientable and spin, any orientable $7$-manifold with spin structure admits a $G_2$-structure, and associated to a given $G_2$-structure $\vp$ are a metric $g_{\vp}$ called the \emph{$G_2$-metric}, satisfying
\begin{equation}
(X\lrcorner\vp)\w(Y\lrcorner\vp)\w\vp=6g_{\vp}(X,Y)\Vol_{\vp}
\label{G2metric}
\end{equation}
for any vector fields $X$ and $Y$ on $M$, a $4$-form 
\begin{equation}
 \sv_0=\d x^{4567}+\d x^{2367}+\d x^{2345}+\d x^{1357}-\d x^{1346}-\d x^{1256}-\d x^{1247}
\label{G24form}
\end{equation}
where $\star$ is the Hodge star operator associated to $g_{\vp}$ and finally a $2$-fold vector cross product $\t$. A natural geometric requirement is that the $3$-form $\vp$ be covariant constant with respect to the Levi-Civita connection of the $G_2$-metric $g_{\vp}$; if this is so, we say that the $G_2$-structure is \emph{integrable} and call the pair $(M,\vp)$ a \emph{$G_2$-manifold}. It is a nontrivial fact that the integrability of the $G_2$-structure is equivalent to the holonomy of $g_{\vp}$ being a subgroup of $G_2$ as well as $\vp$ being simultaneously closed and coclosed, that is, $\d\vp=0$ and $\d^*\vp=0$ respectively, where $\d^*$ is the adjoint to the exterior derivative $\d$ defined in terms of the Hodge star $\star$ of $g_{\vp}$; the condition that $\vp$ be coclosed then implies that the $4$-form $\sv$ is closed. See \cite{Jo1}, \cite{Jo2}, \cite{Kari1}, \cite{Kari2}, \cite{Sa} for more information on these constructions and conditions.

Let $M$ be an $n$-dimensional smooth manifold, and let $\om$ be a closed $(k+1)$-form on $M$ satisfying the nondegeneracy condition 
\begin{equation}
X\lrcorner\om=0\text{ iff  }X=0
\label{nd}
\end{equation}
for $X$ a vector field on $M$. The pair $(M,\om)$ is called a \emph{multisymplectic} $n$-manifold of degree $k+1$; such manifolds can be viewed geometrically as a natural generalization of symplectic manifolds which, in this language, are even-dimensional multisymplectic manifolds of degree $2$. Note that there are the $k$ associated linear maps
\begin{equation}
\begin{array}{cccc}
\widehat{\om_j}:&\Om^j(TM)&\to&\Om^{k+1-j}(T^*M)\\
&Q&\mapsto& Q\lrcorner\om \\
\end{array}
\label{almaps}
\end{equation}
for all $j=1,\ldots,k$ and that the nondegeneracy of $\om$ forces $\widehat{\om_1}$ to be injective and $\widehat{\om_k}$ to be surjective. Also, for multisymplectic forms of degree $\geq 3$, there is a stronger notion of nondegeneracy than that used here, cf. \cite{MaSw2}, \cite{MaSw3}, \cite{MaSw1}. Exact multisymplectic manifolds, that is, multisymplectic manifolds where the multisymplectic $(k+1)$-form is exact, arise naturally in physics as multiphase spaces which are bundles of higher-degree differential forms equipped with an exact multisymplectic form that are generalizations of the standard phase space given by the cotangent bundle equipped with the canonical symplectic form, e. g. \cite{CIdL1}, \cite{CIdL2}, \cite{FPR2}, \cite{FPR1}, \cite{FoRo}. It is important for the purpose of this article to note that Equations \ref{G2metric}, \ref{QsvQvpsveq} imply the nondegeneracy of $\vp$ and $\sv$ in the sense of Equation \ref{nd} which means that a $G_2$-manifold is simultaneously a multisymplectic 
manifold of degree $3$ and of degree $4$ with the multisymplectic $3$-form $\vp$ and multisymplectic $4$-form $\sv$ respectively. 


A vector field $X$ and a real-valued function $H$ on a symplectic manifold $(M,\om)$ satisfying
\begin{equation}
 X\lrcorner\om=\d H
\end{equation}
are referred to as a \emph{Hamiltonian} vector field and a \emph{Hamiltonian} function respectively. Similarly, a $(k-l)$-multivector field $Q$ and a differential $l$-form $\al$ on a multisymplectic manifold $(M,\om)$ of degree $k+1$ satisfying
\begin{equation}
 Q\lrcorner\om=\d\al
\end{equation}
are referred to as a \emph{Hamiltonian} $(k-l)$-multivector field and a \emph{Hamiltonian} $l$-form respectively. The spaces of Hamiltonian vector fields and Hamiltonian functions, along with associated algebraic structures on these spaces, are important and fundamental concepts in symplectic geometry from both the mathematical and the physical perspective and arise from the Hamiltonian formulation of classical mechanics where a Hamiltonian function represents the total energy of a given mechanical system, e. g. \cite{Arno}, \cite{daSi}, \cite{McSa}; similarly, the spaces of Hamiltonian multivector fields and Hamiltonian differential forms, along with associated algebraic structures on these spaces, arise from the covariant Hamiltonian formulation of classical mechanics, and indeed much of the interest in the subject of multisymplectic geometry has come from various areas of physics, e. g. \cite{AtWi}, \cite{BHR}, \cite{BR}, \cite{CIdL2}, \cite{CIdL1}, \cite{CCI}, \cite{FPR2}, \cite{FPR1}, \cite{GYZ}, \cite{
PR}. The work providing the foundations for study of these spaces of Hamiltonian multivector fields and Hamiltonian differential forms on general multisymplectic manifolds has been completed by a number of authors, cf. \cite{CIdL1}, \cite{FPR1}, \cite{PR}.

The purpose of this paper then is to treat $G_2$-geometry as a special case of multisymplectic geometry, consider Hamiltonian multivector fields and Hamiltonian differential forms on manifolds with an integrable $G_2$-structure, called \emph{Rochesterian}/\emph{coRochesterian} multivector fields and \emph{Rochesterian}/\emph{coRochesterian} differential forms to distinguish them from the general multisymplectic setting, and prove some results in multisymplectic geometry that are generalizations of results from symplectic geometry.

This paper comes from a research plan to view $G_2$-geometry using the perspectives of symplectic and contact geometry, cf. \cite{ACS}, \cite{ACS2} for more on contact structures on manifolds with a $G_2$-structure. Treating symplectic geometry and $G_2$-geometry as analogues however is certainly not new. For example, let $V$ be a real, finite-dimensional vector space with inner product $\langle\cdot,\cdot\rangle$ and define a $k$-fold vector cross product to be an alternating, multilinear map
\begin{equation}
 P:\underbrace{V\t\cdots\t V}_{k\text{ times}}\to V
\end{equation}
satisfying
\begin{equation}
 \langle P(v_1,\ldots,v_k), v_i\rangle=0\text{ for all }i=1,\ldots,k
\end{equation}
\begin{equation}
 \langle P(v_1,\ldots,v_k),P(v_1,\ldots,v_k)\rangle=\Vol(v_1,\ldots,v_k)
\end{equation}
where $\Vol$ is the volume form on $V$ with respect to $\langle\cdot,\cdot\rangle$. Using this inner product, there is then also the associated $(k+1)$-exterior form $\om$ on $V$
\begin{equation}
 \om(v_1,\ldots,v_{k+1})=\langle P(v_1,\ldots,v_k),v_{k+1}\rangle
 \label{efatavcp}
\end{equation}
Such objects can be attached to smooth Riemannian manifolds in the standard way by attaching them to the tangent spaces at each point and requiring that they vary smoothly across the manifold; note that a $k$-fold cross product on a Riemannian manifold yields a differential $(k+1)$-form defined point-wise by Equation \ref{efatavcp}. $k$-fold vector cross products on linear spaces are studied by Brown and Gray \cite{BrGr} and on manifolds by Gray \cite{Gray}. They show that a $k$-fold vector cross product on an $n$-dimensional space only exists for certain pairs $(k,n)$: $(n-1,n)$, $(1,2m)$, $(2,7)$ and $(3,8)$. The geometry of an $(n-1)$-fold vector cross product on an $n$-dimensional smooth manifold is simply the Riemannian geometry of the original manifold since the associated differential form has degree $n$ and must therefore be a multiple of the volume form associated to the Riemannian metric; note that the standard two-fold vector cross product on $\R^3$ falls into this category. $1$-fold vector cross 
products are better known as almost complex structures, and the associated differential $2$-form, when closed with respect to the exterior derivative, is then a symplectic form. 

The other two cases arise naturally from the octonions $\O$, an $8$-dimensional, real, non-associative, non-commutative, normed algebra with unit vector $1$ where the $1$-dimensional subspace spanned by the vector $1$ is denoted $\Re\O$ and the $7$-dimensional orthogonal complement is denoted $\Im\O$. Using the algebra multiplication on $\O$, one can define a two-fold vector cross product on $\Im\O$ \cite[Definition $B.1$]{HaLa} and a three-fold vector cross product on $\O$ \cite[Definition $B.1$]{HaLa}. Results of Gray \cite{Gray} show in particular that the existence of a two-fold vector cross product on a $7$-dimensional manifold is equivalent to a reduction in the structure group of the tangent bundle to the exceptional Lie group $G_2$. Work of Fern\'andez and Gray \cite{FeGr} classifies $G_2$-structures into $16$ classes such as closed $G_2$-structures, coclosed $G_2$-structures and integrable $G_2$-structures and, as they mention, should be compared to the various classes of K\"ahler structures; 
examples of manifolds with $G_2$-structures satisfying various of these conditions have been extensively studied, e. g., \cite{Br1}, \cite{Br2}, \cite{BrSa}, \cite{BrXu}, \cite{CMS}, \cite{ClIv}, \cite{Fe1}, \cite{Fe2}, \cite{FeIg}, \cite{Jo1}, \cite{Jo2}, \cite{Sa}. Finally, links between Calabi-Yau geometry and $G_2$-geometry in the context of mirror symmetry have been actively explored, e. g. \cite{AkSa}, \cite{AtWi}, \cite{GYZ}, \cite{LeLe}, \cite{Leun}.

The body of this article consists of two main sections. Section \ref{sectionbackground} is background material for the paper and is itself divided into two subsections. Section \ref{multivectorfields} gives introductory material on multivector fields and associated operations including the Schouten-Nijenhuis bracket and a generalization of \cite[Lemma $A.0.8$]{Kari1} to general $q$-multivector fields 
\begin{lemm}[Lemma \ref{lemmaA08gen}]
For $Q\in\Om^q(TM)$ and $\al\in\Om^l(T^*M)$
\begin{enumerate}
  \item $\star(Q\lrcorner\al)=(-1)^{q(l-q)}Q^{\flat}\w\star\al$
  \item $\star(Q\lrcorner\star\al)=(-1)^{q(n-l-q)+l(n-l)}Q^{\flat}\w\al$
  \item $Q\lrcorner\al=(-1)^{(l-q)(n-l)}\star(Q^{\flat}\w\star\al)$
  \item $Q\lrcorner\star\al=(-1)^{lq}\star(Q^{\flat}\w\al)$
\end{enumerate}
\end{lemm}
\noindent
Section \ref{HSMG} is a discussion of Hamiltonian multivector fields and Hamiltonian differential forms on multisymplectic manifolds largely paralleling the exposition of \cite{CIdL1} together with the multisymplectic version of a well-known result from symplectic geometry 
\begin{pro}[Proposition \ref{PropBracketZero}]
 Let $(M,\om)$ be a multisymplectic manifold of order $(k+1)$, and let $\al_i\in\Om^{k-q_i}_{H}(T^*M)$ with associated Hamiltonian multivector fields $Q_i\in\Om^{q_i}_H(TM)$, $i=1,2$ such that $q_1+q_2=k+1$. Then $\{\al_1,\al_2\}=0$ if and only if $\L_{Q_2}\al_1=0$ (if and only if $\L_{Q_1}\al_2=0$).
\end{pro}
Finally, Section \ref{sectionmain} is the main section of our paper wherein we consider Rochesterian multivector fields and Rochesterian differential forms. We first prove the helpful lemma
\begin{lemm}[Lemma \ref{sv2iso}]
$\widehat{\sv}_2:\Om^2(TM)\to\Om^2(T^*M)$ is an isomorphism.
\end{lemm}
\noindent
Our main results then consist of the following theorem, corollary and proposition.
\begin{thm}
Let $(M,\vp)$ be a $G_2$-manifold.
\begin{equation}
\begin{array}{rcccl}
 \Om^1_R(TM) = & \widetilde{\Om}^1_R(TM)   & \hookrightarrow & \widetilde{\Om}^2_{cR}(TM)   & =\Om^2_{cR}(TM) \\
               &    \updownarrow_{\cong}   &                 & \updownarrow_{\cong}         & \\
               & \widetilde{\Om}^1_R(T^*M) & \subset         & \widetilde{\Om}^1_{cR}(T^*M) & =\Om^1(T^*M)/Z^1(M) \\
\end{array}
\end{equation}
\bigskip
\begin{equation}
\begin{array}{cccl}
 \widetilde{\Om}^2_R(TM)    & \cong & \widetilde{\Om}^0_R(T^*M) &  \\
	\updownarrow_{\cong}&       & \shortparallel & = C^{\infty}(M)/\{f\text{ locally constant}\}\\
 \widetilde{\Om}^3_{cR}(TM) & \cong & \widetilde{\Om}^0_{cR}(T^*M) & \\
\end{array}
\end{equation}
\bigskip
\begin{equation}
 \widetilde{\Om}^1_{cR}(TM)=\Om^1_{cR}(TM)\cong\widetilde{\Om}^2_{cR}(T^*M)
\end{equation}
\label{thmmain}
\end{thm}

\begin{cor}
There are no nonzero Rochesterian vector fields, and hence no nonzero Rochesterian $1$-forms, on a closed manifold $M$ with closed $G_2$-structure $\vp$.
\label{RocVecNonexistence}
\end{cor}

\begin{prop}
Let $(M,\vp)$ be a closed $G_2$-manifold. Then the only $2$-multivector fields $Q$ that are Rochesterian and coRochesterian are those whose contraction with $\vp$ is zero.
\label{RcR2mvfdisjoint}
\end{prop}

\section{Background Material}
\label{sectionbackground}
\subsection{Multivector Fields \& Operations}
\label{multivectorfields}
References for the following material include \cite{MaSw1}, \cite{Marl}; other references are given throughout. Let $V$ be an $n$-dimensional vector space over $\R$ with dual space $V^*$. Using the exterior product $\w$, there are the vector spaces of $l^{th}$-exterior powers $\La^l(V)$ and $\La^l(V^*)$ of $V$ and $V^*$ respectively for $l\geq 0$ whose elements are called \emph{$l$-multivectors} and \emph{exterior $l$-forms} respectively. Explicitly, we can write
\begin{equation}
 \La^l(V)=\sp\left\{v_1\w\cdots\w v_l:v_i\in V\right\}
\end{equation}
\begin{equation}
 \La^l(V^*)=\sp\left\{v^1\w\cdots\w v^l:v^i\in V^*\right\}
\end{equation}
where $\w$ satisfies $v\w w=-w\w v$; this property implies that $\La^l(V)=\La^l(V^*)=\{0\}$ for all $l>n$. The collections of these spaces then forms the associative graded algebras
\begin{equation}
 \La(V)=\bigop_{l\geq 0}\La^l(V)
\end{equation}
and
\begin{equation}
 \La(V^*)=\bigop_{l\geq 0}\La^l(V^*).
\end{equation}

Let $M$ be an $n$-dimensional manifold, so for each $p\in M$, we have the vector spaces given by the tangent space $T_pM$ and the cotangent space $T_p^*M$ which by the above construction yield the $l^{th}$-exterior powers $\La^l(T_pM)$ and $\La^l(T_p^*M)$ respectively. The collection of these spaces gives the associated vector bundles $\La^l(TM)$ and $\La^l(T^*M)$. Smooth sections of these bundles, denoted by $\Om^l(TM)$ and $\Om^l(T^*M)$, are called \emph{$l$-multivector fields} and \emph{differential $l$-forms} respectively where $\Om^0(TM)=\Om^0(T^*M)=C^{\infty}(M)$ are smooth real-valued functions on $M$, $\Om^1(TM)$ is the space of vector fields on $M$, $\Om^l(TM)\cong\Om^l(T^*M)\cong\{0\}$ for all $l>n$; an $l$-multivector field that can be written as $X_1\w\cdots\w X_l$ for $X_1,\ldots,X_l\in\Om^1(TM)$ is called \emph{decomposable}. As above, we can take the collection of multivector fields and differential forms for all $l\geq 0$ to get the associative graded algebras
\begin{equation}
 \Om(TM)=\bigop_{l\geq0}\Om^l(TM)
\end{equation}
and
\begin{equation}
 \Om(T^*M)=\bigop_{l\geq0}\Om^l(T^*M)
\end{equation}
of multivector fields and differential forms respectively. 

For a differential $\tilde{l}$-form $\phi$ and a decomposable $l$-multivector $X_1\w\cdots\w X_l$ define the contraction of $\phi$ by $X_1\w\cdots\w X_l$
\begin{equation}
 (X_1\w\cdots\w X_l)\lrcorner\phi=X_l\lrcorner\cdots\lrcorner X_1\lrcorner\phi
\end{equation}
where $\lrcorner$ denotes the standard contraction of a differential form by a vector field; if $f\in\Om^0(TM)$, then contraction reduces to multiplication 
\begin{equation}
f\lrcorner\phi=f\phi,
\end{equation}
and for a $q$-multivector $Q$ with $q>\tilde{l}$, 
\begin{equation}
 Q\lrcorner\phi=0.
\end{equation}

We next consider the Schouten-Nijenhuis bracket on the collection of multivector fields, a natural extension of the standard Lie bracket on vector fields. The construction given here follows that of \cite{Va}; see \cite[Appendix A]{FPR2} for an equivalent, up to signs, perspective on this bracket and related operations. For $X\in\Om^1(TM)$ and $Q\in\Om^q(TM)$ define an extension of the standard Lie derivative by
\begin{equation}
\left(\L_XQ\right)(p)=\left.\frac{d}{dt}\right|_{t=0}\left(\exp(-tX)_*Q_{\exp(tX)(p)}\right)
\end{equation}
where $p\in M$ and $\{\exp(tX)\}$ denotes the one-parameter subgroup of diffeomorphisms generated by $X$. For a decomposable $l$-multivector field $X_1\w\cdots\w X_{l}$ then define
\begin{equation}
[X_1\w\cdots\w X_{l},Q]=\sum_{i=1}^{l}(-1)^{i+1}X_1\w\cdots\w \hat{X}_i\w\cdots\w X_{l}\w[X_i,Q]
\end{equation}
where $[X_i,Q]=\L_{X_i}Q$. Extend this operation to a general multivector field by linearity, and note that for $Q_i\in\Om^{q_i}(TM)$, $i=1,2$, we have $[Q_1,Q_2]\in\Om^{q_1+q_2-1}(TM)$. This bracket, called the \emph{Schouten-Nijenhuis bracket}, satisfies the following properties
\begin{equation}
[Q_1,Q_2]=(-1)^{q_1q_2}[Q_2,Q_1]
\end{equation}
\begin{equation}
[Q_1,Q_2\w Q_3]=[Q_1,Q_2]\w Q_3 + (-1)^{q_1q_2+q_2}Q_2\w[Q_1,Q_3]
\end{equation}
\begin{equation}
\begin{split}
(-1)^{q_1(q_3-1)}&[Q_1,[Q_2,Q_3]]+(-1)^{q_2(q_1-1)}[Q_2,[Q_3,Q_1]] \\
&+(-1)^{q_3(q_2-1)}[Q_3,[Q_1,Q_2]]=0 \\
\end{split}
\end{equation}
for $Q_i\in\Om^{q_i}(TM)$, $i=1,2,3$. The Schouten-Nijenhuis bracket gives the algebra of multivector fields $\Om(TM)$ the structure of a \emph{Gerstenhaber algebra}, cf. \cite{Marl}, \cite{VYL}, \cite{Xu}. 

To describe the relationship between the Schouten-Nijenhuis bracket and the contraction operation, we use an extension of the standard Lie derivative of a differential form along a vector field to that along a multivector field. Our sign conventions follow those given by \cite{Mich}. For $Q\in\Om^q(TM)$ and $\al\in\Om^l(T^*M)$ define
\begin{equation}
 \L_Q\al=Q\lrcorner\d\al-(-1)^q\d(Q\lrcorner\al)
\label{LQal}
\end{equation}
This operator satisfies the following properties
\begin{equation}
 \d\L_Q\al=(-1)^{q+1}\L_Q\d\al
\end{equation}
\begin{equation}
 \L_{Q_1\w Q_2}\al=Q_2\lrcorner\L_{Q_1}\al+(-1)^{q_1}\L_{Q_2}Q_1\lrcorner\al
\end{equation}
\begin{equation}
 [Q_1,Q_2]\lrcorner\al=(-1)^{q_1q_2+q_2}\L_{Q_1}Q_2\lrcorner\al-Q_2\lrcorner\L_{Q_1}\al
\label{SNbipeLd}
\end{equation}
for all $Q\in\Om^q(TM)$, $Q_i\in\Om^{q_i}(TM)$, $i=1,2$, $\al\in\Om^l(T^*M)$.

Finally, there are a number of relations that are straightforward generalizations to multivector fields of those relations given in \cite[Lemma $A.0.8$]{Kari1} for vector fields that will be useful in Section \ref{sectionmain}. Let $M$ be an $n$-dimensional Riemannian manifold with Riemannian metric $g$ where $g$ is used to also denote the extension of Riemannian metric to the spaces of $j$-multivectors and differential $j$-forms. Let $\flat:\Om^j(TM)\to\Om^j(T^*M)$ be the isomorphism such that 
\begin{equation}
 g(Q,Q)=Q^{\flat}(Q)=g(Q^{\flat},Q^{\flat})\text{ for all }Q\in\Om^j{TM)},
\label{flat}
\end{equation}
$\sharp:\Om^j(T^*M)\to\Om^j(TM)$ the corresponding isomorphism such that
\begin{equation}
 g(\al^{\sharp},\al^{\sharp})=\al(\al^{\sharp})=g(\al,\al)\text{ for all }\al\in\Om^j(T^*M)
\label{sharp}
\end{equation}
and $\star$ the Hodge star isomorphism associated to the Riemannian metric $g$ where, for $\al,\be\in\Om^l(T^*M)$, $\star$ satisfies
\begin{equation}
 \be\w\star\al=g(\be,\al)\Vol_M=g(\al,\be)\Vol_M=\al\w\star\be
\end{equation}

\begin{lem}
For $Q\in\Om^q(TM)$ and $\al\in\Om^l(T^*M)$
\begin{enumerate}
  \item $\star(Q\lrcorner\al)=(-1)^{q(l-q)}Q^{\flat}\w\star\al$
  \item $\star(Q\lrcorner\star\al)=(-1)^{q(n-l-q)+l(n-l)}Q^{\flat}\w\al$
  \item $Q\lrcorner\al=(-1)^{(l-q)(n-l)}\star(Q^{\flat}\w\star\al)$
  \item $Q\lrcorner\star\al=(-1)^{lq}\star(Q^{\flat}\w\al)$
\end{enumerate}
\label{lemmaA08gen}
\end{lem}

\begin{proof}
For $\be\in\Om^{l-q}(T^*M)$,
\begin{equation}
\begin{split}
\be&\w\star(Q\lrcorner\al)=g(\be,Q\lrcorner\al)\Vol_M \\
&=(Q\lrcorner\al)(\be^{\sharp})\Vol_M=\al(Q\w\be^{\sharp})\Vol_M \\
&=g(\al,Q^{\flat}\w\be)\Vol_M=(Q^{\flat}\w\be)\w\star\al \\
&=(-1)^{q(l-q)}\be\w Q^{\flat}\w\star\al \\
\end{split}
\end{equation}
from which we get the relation
\begin{equation}
\star(Q\lrcorner\al)=(-1)^{q(l-q)}Q^{\flat}\w\star\al
\label{Qflat1}
\end{equation}
Recall that $\star\star\al=(-1)^{l(n-l)}\al$ which, together with Equation \ref{Qflat1}, implies that
\begin{equation}
\star(Q\lrcorner\star\al)=(-1)^{q(n-l-q)+l(n-l)}Q^{\flat}\w\al
\label{Qflat2}
\end{equation}
Taking $\star$ of Equations \ref{Qflat1}, \ref{Qflat2} gives
\begin{equation}
Q\lrcorner\al=(-1)^{(l-q)(n-l)}\star(Q^{\flat}\w\star\al)
\label{Qflat3}
\end{equation}
\begin{equation}
Q\lrcorner\star\al=(-1)^{lq}\star(Q^{\flat}\w\al)
\label{Qflat4}
\end{equation}
\end{proof}

\subsection{Hamiltonian Structures in Multisymplectic Geometry}
\label{HSMG}
Our presentation parallels that of \cite{CIdL1}; \cite{FPR1} also covers this material using coordinates on multiphase spaces. Let $(M,\om)$ be an $n$-dimensional multisymplectic manifold of degree $k+1$. Given a $(k-l)$-multivector field $Q\in\Om^{k-l}(TM)$ with $0\leq l\leq k-1$, we say that $Q$ is \emph{locally Hamiltonian} if the $(l+1)$-form $Q\lrcorner\om$ is closed; we say that $Q$ is \emph{Hamiltonian} if there is an $l$-form $\al$ such that 
\begin{equation}
 Q\lrcorner\om=\d\al
\label{ham}
\end{equation}
An $l$-form $\al$ is \emph{Hamiltonian} if there is a $(k-l)$-multivector field $Q$ satisfying Equation \ref{ham}. We denote the collection of locally Hamiltonian $(k-l)$-multivector fields by $\Om^{k-l}_{lH}(TM)$, the collection of Hamiltonian $(k-l)$-multivector fields by $\Om^{k-l}_{H}(TM)$ and the collection of Hamiltonian $l$-forms by $\Om^l_H(T^*M)$. A number of properties are immediately evident from these definitions. $\d^2=0$ gives
\begin{equation}
 \Om^{k-l}_{H}(TM)\subseteq\Om^{k-l}_{lH}(TM)
\end{equation}
for each $l$, and if $H^{l+1}_{dR}(M)=\{0\}$, that is, if the $(l+1)$-de Rham cohomology space of $M$ is trivial, then $\Om^{k-l}_H(TM)=\Om^{k-l}_{lH}(TM)$. Linearity of the exterior derivative $\d$ and the interior product $\lrcorner$ gives that $\Om^{k-l}_{lH}(TM)$, $\Om^{k-l}_{H}(TM)$ and $\Om^l_H(T^*M)$ are vector spaces over $\R$, but in contrast to $\Om^{k-l}(TM)$ and $\Om^{l}(T^*M)$, we cannot view any of $\Om^{k-l}_{lH}(TM)$, $\Om^{k-l}_{H}(TM)$, $\Om^l_H(T^*M)$, in general, as modules over the commutative ring $C^{\infty}(M)$ of real-valued functions on $M$; moreover, while we can consider the associated graded vector spaces $\Om_{lH}(TM)$, $\Om_{H}(TM)$, $\Om_H(T^*M)$, these spaces are not in general closed under the wedge product $\w$, so they cannot be considered as subalgebras of multivector fields/differential forms. 

Notice that, given $Q\in\Om^{k-l}_H(TM)$, the associated Hamiltonian $l$-form $\al_Q\in\Om^l_H(T^*M)$ is only defined up to the addition of a closed $l$-form. Let $Z^l(M)$ denote the space of closed $l$-forms which are a subspace of $\Om^l_H(T^*M)$ and consider the quotient space $\widetilde{\Om}^l_H(T^*M)$ given by 
\begin{equation}
 \widetilde{\Om}^l_H(T^*M)=\Om^l_H(T^*M)/Z^l(M)
\end{equation}
Thus, given $Q\in\Om^{k-l}_H(TM)$, there exists a unique $\tilde{\al}\in\widetilde{\Om}^l_H(T^*M)$ such that $Q\lrcorner\om=\d\tilde{\al}$. Conversely, given $\al\in\Om^l_H(T^*M)$, the associated Hamiltonian $(k-l)$-multivector field $Q_{\al}\in\Om^{k-l}(TM)$ is only defined up to a $(k-l)$-multivector field whose contraction with $\om$ is zero. Such $(k-1)$-multivector fields form a subspace of $\Om^{k-l}_H(TM)$, so we can consider the quotient space $\widetilde{\Om}^{k-l}_H(TM)$. These considerations yield the isomorphism of vector spaces
\begin{equation}
 \widetilde{\Om}^l_H(T^*M)\cong\widetilde{\Om}^{k-l}_H(TM)\text{ for all }0\leq l\leq k-1
\end{equation}
In the specific case of $l=k-1$, injectivity of $\widehat{\om}_1$ implies that the Hamiltonian $1$-multivector field, i. e., Hamiltonian vector field, associated to a given Hamiltonian $(k-1)$-form is uniquely determined by that form, that is,
\begin{equation}
 \widetilde{\Om}^{k-1}_H(T^*M)\cong\widetilde{\Om}^{1}_H(TM)\cong\Om^1_H(TM)
\end{equation}
and in the case of $l=0$, surjectivity of $\widehat{\om}_{k}$ implies that
\begin{equation}
 \Om^0_H(T^*M)=C^{\infty}(M)
\end{equation}
and hence that
\begin{equation}
 \widetilde{\Om}^0_{H}(T^*M)=C^{\infty}(M)/\{f:M\to\R|f\text{ is locally constant}\}.
\end{equation}

\begin{rem}
The injectivity of $\widehat{\om}_1$ implies that there exists an injective map
\begin{equation}
 \Om^1_{lH}(TM)\hookrightarrow Z^k(M)
\end{equation}
and the surjectivity of $\widehat{\om}_k$ implies that
\begin{equation}
 \widetilde{\Om}^k_{lH}(TM)\cong Z^1(M)
\end{equation}
where $\widetilde{\Om}^{k}_{lH}(TM)$ denotes the quotient space of $\Om^{k}_{lH}(TM)$ by the subspace of all elements of $\Om^{k}(TM)$ whose contraction with $\om$ is zero.
\end{rem}

Next, let $Q_i\in\Om^{q_i}_{lH}(TM)$ for $i=1,2$. Then
\begin{equation}
\begin{split}
 [Q_1,Q_2]\lrcorner\om&=(-1)^{q_1q_2+q_2}\L_{Q_1}(Q_2\lrcorner\om)-Q_2\lrcorner(\underbrace{\L_{Q_1}\om}_{=0}) \\
 &=(-1)^{q_1q_2+q_2}\L_{Q_1}(Q_2\lrcorner\om) \\
 &=(-1)^{q_1q_2+q_2}(Q_1\lrcorner\underbrace{\d(Q_2\lrcorner\om)}_{=0}-(-1)^{q_1}\d(Q_1\lrcorner Q_2\lrcorner\om)) \\
 &=(-1)^{q_1q_2+q_2+q_1+1}\d(Q_2\w Q_1\lrcorner\om) \\
 &=(-1)^{q_1+q_2+1}\d(Q_1\w Q_2\lrcorner\om) \\
\end{split}
\label{snh}
\end{equation}
shows $[Q_1,Q_2]\in\Om^{q_1+q_2-1}_{H}(TM)$. If the contraction of $\om$ by either $Q_1$ or $Q_2$ is zero, then $[Q_1,Q_2]\lrcorner\om=0$ showing that $[\cdot,\cdot]$ is well-defined on the quotient spaces $\widetilde{\Om}^{k-l}_H(TM)$. Also, with $q_1=q_2=1$, this shows $(\Om^1_{H}(TM),[\cdot,\cdot])$ and $(\Om^1_{lH}(TM),[\cdot,\cdot])$ are Lie subalgebras of the Lie algebra of vector fields $(\Om^1(TM),[\cdot,\cdot])$. 

Let $\al_i\in\Om^{k-q_i}_H(T^*M)$ with associated Hamiltonian $q_i$-multivector fields $Q_i$ for $i=1,2$. Define a bracket operation on $\Om_H(T^*M)$ by
\begin{equation}
 \{\al_1,\al_2\}=(-1)^{q_1+q_2+1}(Q_1\w Q_2)\lrcorner\om
\label{hfb}
\end{equation}
with $\{\al_1,\al_2\}=0$ if $q_1+q_2>k+1$. This definition is independent of the choices of associated Hamiltonian multivector field and satisfies
\begin{equation}
\begin{split}
 \{\al_2,\al_1\}&=(-1)^{q_2+q_1+1}(Q_2\w Q_1)\lrcorner\om \\
 &=(-1)^{q_1q_2}(-1)^{q_1+q_2+1}(Q_1\w Q_2)\lrcorner\om=(-1)^{q_1q_2}\{\al_1,\al_2\}\\
\end{split}
\end{equation}
Further, $\{\al_1,\al_2\}$ is a Hamiltonian $(k+1-q_1-q_2)$-form because 
\begin{equation}
 \d\{\al_1,\al_2\}=(-1)^{q_1+q_2+1}\d(Q_1\w Q_2\lrcorner\om)=[Q_1,Q_2]\lrcorner\om
\end{equation}
This bracket is well-defined on the quotient spaces $\widetilde{\Om}^{l}_H(T*M)$ because if, for example, we assume that $\d\al_1=0$, then by Equations \ref{ham}, \ref{hfb},
\begin{equation}
\begin{split}
 \{\al_1,\al_2\}&=(-1)^{q_1+q_2+1}(Q_1\w Q_2\lrcorner\om)\\
 &=(-1)^{q_1+q_2+1}(Q_2\lrcorner Q_1\lrcorner\om)\\
 &=(-1)^{q_1+q_2+1}(Q_2\lrcorner\d\al_1)=0
\end{split}
\end{equation}
A similar argument can also be applied in the case that the second entry is closed to get the claim; then define
\begin{equation}
 \{\tilde{\al_1},\tilde{\al_2}\}=\widetilde{\{\al_1,\al_2\}}
\end{equation}

Note that the above means that we need to take the Lie degree of $\al\in\Om^{k-q}_H(T^*M)$ to be $q$, that is, the degree of $\al\in\Om^{k-q}_H(T^*M)$ as an element of the graded Lie algebra of Hamiltonian forms is defined as $k$ minus the tensor degree of $\al$; the Lie degree of $\{\al_1,\al_2\}$ is then $k-(k+1-q_1-q_2)=q_1+q_2-1$. We last consider the Jacobi identity for this bracket, so let $\al_i\in\Om^{k-q_i}_{H}(T^*M)$ with associated Hamiltonian $q_i$-multivector fields $Q_i=Q_{\al_i}\in\Om^{q_i}_H(TM)$ for $i=1,2,3$. 
\begin{equation}
\begin{split}
 &(-1)^{q_1(q_3-1)}\{\al_1,\{\al_2,\al_3\}\} \\
=&(-1)^{q_1(q_3-1)}(-1)^{q_1+q_2+q_3}(Q_1\w [Q_2,Q_3])\lrcorner\om \\
=&(-1)^{q_1(q_3-1)}(-1)^{q_1+q_2+q_3}(-1)^{q_1(q_2+q_3-1)}(-1)^{q_2q_3}([Q_3,Q_2]\w Q_1)\lrcorner\om \\
=&(-1)^{q_1+q_2+q_3+q_1q_2+q_2q_3}Q_1\lrcorner [Q_3,Q_2]\lrcorner\om \\
=&(-1)^{q_1+q_2+q_3+q_1q_2+q_2q_3}(-1)^{q_3+q_2+1}Q_1\lrcorner\d((Q_3\w Q_2)\lrcorner\om) \\
=&(-1)^{q_1+q_1q_2+q_2q_3+1}Q_1\lrcorner\d(Q_2\lrcorner Q_3\lrcorner\om) \\
=&(-1)^{q_1+q_1q_2+q_2q_3+1}Q_1\lrcorner\d(Q_2\lrcorner\d\al_3) \\
\end{split}
\end{equation}
Similarly,
\begin{equation}
\begin{split}
 &(-1)^{q_2(q_1-1)}\{\al_2,\{\al_3,\al_1\}\}=(-1)^{q_2+q_2q_3+1}Q_2\lrcorner\d(Q_1\lrcorner\d\al_3) \\
\end{split}
\end{equation}
Finally, with the use of Equations \ref{LQal}, \ref{SNbipeLd}
\begin{equation}
\begin{split}
 &(-1)^{q_3(q_2-1)}\{\al_3,\{\al_1,\al_2\}\} \\
=&(-1)^{q_3(q_2-1)}(-1)^{q_3+q_1+q_2}(Q_3\w [Q_1,Q_2])\lrcorner\om \\
=&(-1)^{q_3q_2+q_1+q_2}[Q_1,Q_2]\lrcorner Q_3\lrcorner\om \\ 
=&(-1)^{q_3q_2+q_1+q_2}[Q_1,Q_2]\lrcorner\d\al_3 \\
=&(-1)^{q_3q_2+q_1+q_2}((-1)^{q_1q_2+q_2}\L_{Q_1}(Q_2\lrcorner\d\al_3)-Q_2\lrcorner\L_{Q_1}\d\al_3) \\
=&(-1)^{q_3q_2+q_1+q_1q_2}\L_{Q_1}(Q_2\lrcorner\d\al_3)-(-1)^{q_3q_2+q_1+q_2}Q_2\lrcorner\L_{Q_1}\d\al_3) \\
=&(-1)^{q_3q_2+q_1+q_1q_2}(Q_1\lrcorner\d(Q_2\lrcorner\d\al_3)-(-1)^{q_1}\d(Q_1\lrcorner\d(Q_2\lrcorner\d\al_3))) \\ &-(-1)^{q_3q_2+q_1+q_2}Q_2\lrcorner(Q_1\lrcorner\d(\d\al_3)-(-1)^{q_1}\d(Q_1\lrcorner\d\al_3)) \\
=&(-1)^{q_3q_2+q_1+q_1q_2}Q_1\lrcorner\d(Q_2\lrcorner\d\al_3)-(-1)^{q_3q_2+q_1q_2}\d(Q_1\lrcorner\d(Q_2\lrcorner\d\al_3)) \\
&+(-1)^{q_3q_2+q_2}Q_2\lrcorner\d(Q_1\lrcorner\d\al_3)) \\
\end{split}
\end{equation}
These equations combine to yield the relation
\begin{equation}
 \begin{split}
  (-1)&^{q_1(q_3-1)}\{\al_1,\{\al_2,\al_3\}\} +(-1)^{q_2(q_1-1)}\{\al_2,\{\al_3,\al_1\}\} \\
  &+(-1)^{q_3(q_2-1)}\{\al_3,\{\al_1,\al_2\}\}=(-1)^{q_3q_2+q_1q_2+1}\d(Q_1\lrcorner\d(Q_2\lrcorner\d\al_3))
 \end{split}
\end{equation}
which is zero in the quotient space $\widetilde{\Om}_H(T^*M)$. This bracket and grading thus give the quotient space $\widetilde{\Om}_H(T^*M)$ the structure of a graded Lie algebra which can be identified with the graded Lie algebra of the quotient space $\widetilde{\Om}_H(TM)$ together with the Schouten-Nijenhuis bracket.

\begin{rem}
This bracket is a generalization of the \emph{semibracket} defined in \cite{BHR} for the specific case of $q_1=q_2=1$, and the proof of the previous result for that case, of which our proof is a straightforward generalization, can be found in \cite[Proposition 3.7]{BHR}.
\end{rem}

\begin{prop}
 Let $(M,\om)$ be a multisymplectic manifold of order $(k+1)$, and let $\al_i\in\Om^{k-q_i}_{H}(T^*M)$ with associated Hamiltonian multivector fields $Q_i\in\Om^{q_i}_H(TM)$, $i=1,2$ such that $q_1+q_2=k+1$. Then $\{\al_1,\al_2\}=0$ if and only if $\L_{Q_2}\al_1=0$ (if and only if $\L_{Q_1}\al_2=0$).
 \label{PropBracketZero}
\end{prop}

\begin{proof}
Note that $q_1+q_2=k+1$ implies $q_2=k+1-q_1>k-q_1$; in particular, this means that $Q_2\lrcorner\al=0$ which, together with Equation \ref{LQal} gives \begin{equation}
\begin{split}
   \{\al_1,\al_2\}&=(-1)^{q_1+q_2+1}Q_1\w Q_2\lrcorner\om=(-1)^{q_1+q_2+1}Q_2\lrcorner Q_1\lrcorner\om \\
   &=(-1)^{q_1+q_2+1}Q_2\lrcorner\d\al_1=(-1)^{q_1+q_2+1}\L_{Q_2}\al_1 \\
\end{split}
\end{equation}
\end{proof}

We remark that in the case $q_1+q_2<k+1$, then the above calculation yields 
\begin{equation}
 \{al_1,\al_2\}=(-1)^{q_1+q_2+1}\L_{Q_2}\al_1+(-1)^{q_1+q}\d(Q_2\lrcorner\d\al_1)
\end{equation}
showing only that $\widetilde{\{\al_1,\al_2\}}\in\widetilde{\Om}^{k+1-q_1-q_2}_H(T^*M)$ is zero if and only if $\widetilde{\L_{Q_2}\al_1}\in\widetilde{\Om}^{k+1-q_1-q_2}_H(T^*M)$ is zero.

\section{(Co)Rochesterian Multivector Fields and (Co)Rochesterian Differential Forms on $G_2$-Manifolds}
\label{sectionmain}
\begin{dfn}
Let $(M,\vp)$ be a manifold with closed $G_2$-structure.
\begin{enumerate}
    \item An $l$-multivector field $Q$ is called a \emph{$G_2$ $l$-multivector field}, $l=1,2$, if the $(3-l)$-form $Q\lrcorner\vp$ is closed; $\Om^l_{G_2}(TM)$ will denote the set of $G_2$ $l$-multivector fields on $M$.
    \item An $l$-multivector field $Q$ is called a \emph{Rochesterian $l$-multivector field}, $l=1,2$, if there is a $(2-l)$-form $\al$ satisfying $Q\lrcorner\vp=\d\al$; $\Om^l_{R}(TM)$ will denote the set of Rochesterian $l$-multivector fields on $M$.
    \item A differential $l$-form $\al$ is called a \emph{Rochesterian $l$-form}, $l=0,1$, if there is a $(2-l)$-multivector field $Q$ satisfying $\d\al=Q\lrcorner\vp$; $\Om^l_{R}(T^*M)$ will denote the set of Rochesterian $l$-forms on $M$.
\end{enumerate}
\end{dfn}

Corresponding to the multisymplectic $7$-manifold $(M,\vp)$ of degree $3$ there are the spaces
\begin{equation}
 \Om_{G_2}(TM)=\Om^1_{G_2}(TM)\op\Om^2_{G_2}(TM)
\end{equation}
\begin{equation}
 \Om_{R}(TM)=\Om^1_{R}(TM)\op\Om^2_{R}(TM)
\end{equation}
\begin{equation}
 \Om_{R}(T^*M)=\Om^0_{R}(T^*M)\op\Om^1_{R}(T^*M)
\end{equation}
together with the quotient spaces as in Section \ref{HSMG}; moreover, there are the linear maps as in Equation \ref{almaps}
\begin{equation}
 \widehat{\vp}_1:\Om^1(TM)\to\Om^2(T^*M)
\end{equation}
\begin{equation}
\widehat{\vp}_2:\Om^2(TM)\to\Om^1(T^*M)
\end{equation}
with $\widehat{\vp}_1$ injective and $\widehat{\vp}_2$ surjective.

\begin{dfn}
Let $(M,\vp)$ be a manifold with coclosed $G_2$-structure.
\begin{enumerate}
    \item An $l$-multivector field $Q$ is called a \emph{co$G_2$ $l$-multivector field}, $l=1,2,3$, if the $(4-l)$-form $Q\lrcorner\sv$ is closed; $\Om^l_{cG_2}(TM)$ will denote the set of co$G_2$ $l$-multivector fields on $M$.
    \item An $l$-multivector field $Q$ is called a \emph{coRochesterian $l$-multivector field}, $l=1,2,3$, if there is a $(3-l)$-form $\al$ satisfying $Q\lrcorner\sv=\d\al$; $\Om^l_{cR}(TM)$ will denote the set of coRochesterian $l$-multivector fields on $M$.
    \item A differential $l$-form $\al$ is called a \emph{coRochesterian $l$-form}, $l=0,1,2$, if there is a $(3-l)$-vector field $Q$ satisfying $\d\al=Q\lrcorner\vp$; $\Om^l_{R}(T^*M)$ will denote the set of Rochesterian $l$-forms on $M$.
\end{enumerate}
\end{dfn}

Corresponding to the multisymplectic $7$-manifold $(M,\sv)$ of degree $4$ there are the spaces
\begin{equation}
 \Om_{cG_2}(TM)=\Om^1_{cG_2}(TM)\op\Om^2_{cG_2}(TM)\op\Om^3_{G_2}(TM)
\end{equation}
\begin{equation}
 \Om_{cR}(TM)=\Om^1_{cR}(TM)\op\Om^2_{cR}(TM)\op\Om^3_{cR}(TM)
\end{equation}
\begin{equation}
 \Om_{cR}(T^*M)=\Om^0_{cR}(T^*M)\op\Om^1_{cR}(T^*M)\op\Om^2_{cR}(T^*M)
\end{equation}
together with the quotient spaces as in Section \ref{HSMG}; moreover, there are the linear maps as in Equation \ref{almaps}
\begin{equation}
 \widehat{\sv}_1:\Om^1(TM)\to\Om^3(T^*M)
\end{equation}
\begin{equation}
\widehat{\sv}_2:\Om^2(TM)\to\Om^2(T^*M)
\end{equation}
\begin{equation}
 \widehat{\sv}_3:\Om^3(TM)\to\Om^1(T^*M)
\end{equation}
with $\widehat{\vp}_1$ injective and $\widehat{\vp}_3$ surjective.

\begin{lem}
$\widehat{\sv}_2:\Om^2(TM)\to\Om^2(T^*M)$ is an isomorphism.
\label{sv2iso}
\end{lem}

\begin{proof}
Recall that the action of $G_2$ on $\R^7$ induces an action of $G_2$ on spaces of differential forms on a manifold with $G_2$-structure, so we can decompose each space of $k$-forms into irreducible $G_2$-representations. References for this material include \cite{FeGr}, \cite{Jo1} and \cite{Sa}. In particular, as in \cite{Kari1}, it is the case that
\begin{equation}
 \Om^2(T^*M)=\Om^2_7(T^*M)\op\Om^2_{14}(T^*M)
\label{2formdecomp}
\end{equation}
where $\Om^2_l$ denote the $l$-dimensional irreducible $G_2$-representation in $\Om^2$ and
\begin{equation}
\Om^2_7=\{\be\in\Om^2:\star(\vp\w\be)=2\be\}
\label{O27}
\end{equation}
\begin{equation}
\Om^2_{14}=\{\be\in\Om^2:\star(\vp\w\be)=-\be\}
\label{O214}
\end{equation}
These descriptions then yield the following relations for $\be\in\Om^2(T^*M)$
\begin{equation}
 \star(\vp\w\be)=2\pi_7(\be)-\pi_{14}(\be)
\label{svpbe}
\end{equation}
\begin{equation}
 \pi_7(\be)=\frac{\be+\star(\vp\w\be)}{3}
\label{svpbe7}
\end{equation}
 \begin{equation}
 \pi_{14}(\be)=\frac{2\be-\star(\vp\w\be)}{3}
\label{svpbe14}
\end{equation}
with 
\begin{equation}
 \pi_7:\Om^2(T^*M)\to\Om^2_7(T^*M)
\end{equation}
\begin{equation}
 \pi_{14}:\Om^2(T^*M)\to\Om^2_{14}(T^*M)
\end{equation}
the natural projection maps.

If $\widehat{\sv}_2(Q)=Q\lrcorner\sv=0$, then $\star(\vp\w Q^{\flat})=0$ by Equation \ref{Qflat4}; this, together with Equations \ref{svpbe}, \ref{svpbe7}, \ref{svpbe14}, implies that
\begin{equation}
 2\pi_7(Q^{\flat})=\pi_{14}(Q^{\flat})=\frac{2}{3}Q^{\flat}
\end{equation}
showing that $Q^{\flat}\in\Om^2_7\cap\Om^2_{14}=\{0\}$. Thus $Q=0$, and hence $\widehat{\sv}_2$ is injective. Now, let $\be\in\Om^2(T^*M)$, so we can write $\be=\be_7+\be_{14}$ by Equation \ref{2formdecomp} where $\be_7=\pi_7(\be)$ and $\be_{14}=\pi_{14}(\be)$. Using Equations \ref{O27}, \ref{O214}, \ref{Qflat4} and the map $\sharp$ of Equation \ref{sharp}
\begin{equation}
 \be_7=\frac{1}{2}\star(\vp\w\be_7)=\frac{1}{2}\star(\be_7\w\vp)=\frac{1}{2}\be_7^{\sharp}\lrcorner\sv
\end{equation}
\begin{equation}
 \be_{14}=-\star(\vp\w\be_{14})=-\star(\be_{14}\w\vp)=-\be_{14}^{\sharp}\lrcorner\sv
\end{equation}
from which it follows that
\begin{equation}
 \be=\be_7+\be_{14}=(\frac{1}{2}\be_7^{\sharp}-\be_{14}^{\sharp})\lrcorner\sv
\end{equation}
Thus $\widehat{\sv}_2$ is surjective and hence bijective.
\end{proof}

\begin{cor}
 \begin{equation}
\widetilde{\Om}^2_{cR}(TM)=\Om^2_{cR}(TM)
\end{equation}
\begin{equation}
 \Om^1_{cR}(T^*M)=\Om^1(T^*M)
\end{equation}
\begin{equation}
\Om^2_{cG_2}(TM)\cong Z^2(M)
\end{equation}
\label{corOm2cR}
\end{cor}

\begin{cor}
 \begin{equation}
 \Om^1_R(T^*M)\subset\Om^1_{cR}(T^*M).
\label{R1cR1}
 \end{equation}
\end{cor}

The reverse inclusion does not hold in general. For example, consider $(\R^7,\vp_0)$ where $\vp_0$ is the $3$-form defined in Equation \ref{G23form}, let $Q$ be the $2$-multivector field given by $Q=\frac{\pd}{\pd x^6}\w\frac{\pd}{\pd x^6}$ and let $\al$ be the $1$-form defined by $\al=x^4\d x^5+x^2\d x^3$. Then $\al$ is coRochesterian since
\begin{equation}
 Q\lrcorner\sv_0=\d x^{45}+\d x^{23}=\d\al
\end{equation}
where $\sv_0$ is the $4$-form defined in Equation \ref{G24form}; however, $\al$ is not Rochesterian since for a general vector field $X=\sum_{i=1}^7X_i\frac{\pd}{\pd x^i}$,
\begin{equation}
\begin{split}
 X\lrcorner\vp_0&=X_1(\d x^{23}+\d x^{45}+\d x^{67})+X_2(-\d x^{13}+\d x^{46}-\d x^{57}) \\
 &+X_3(\d x^{12}-\d x^{47}-\d x^{56})+X_4(-\d x^{15}-\d x^{26}+\d x^{37}) \\
 &+X_5(\d x^{14}+\d x^{27}+\d x^{36})+X_6(-\d x^{17}+\d x^{24}-\d x^{35})  \\
 &+X_7(\d x^{16}-\d x^{25}-\d x^{34})
\end{split}
\end{equation}
which can never be equal to $\d x^{45}+\d x^{23}$. This example also emphasizes the fact that the space of Rochesterian multivector fields is not in general closed under the wedge product since on $(\R^7,\vp_0)$, straightforward calculations show that every coordinate vector field is a Rochesterian vector field. From \cite{CST}, if $X$ is a $3$-dimensional manifold, then $(T^*X\t\R,\vp=\Re\Om+\om\w\d t)$ is a $7$-manifold with closed $G_2$-structure where $\Om$ is a certain complex $3$-form and $\om$ is the tautological $2$-form on $T^*X$. Then the vector field $\frac{\pd}{\pd t}$ is Rochesterian with an associated Rochesterian $1$-form given by the tautological $1$-form $\al$ on $T^*X$; however, because $\widehat{\vp}_1$ is not surjective, the existence of nonzero Rochesterian vector fields, and hence Rochesterian $1$-forms, on manifolds with a closed $G_2$-structure is not guaranteed as is seen as a consequence of the following theorem (see \cite[Theorem 2.4]{ACS} for the original form of the statement and 
proof of this theorem).

\begin{lem}
Let $M$ be a closed manifold, and let $\vp$ be a closed $G_2$-structure on $M$. Then $X\lrcorner\vp$ is exact if and only if $X$ is the zero vector field.
\end{lem}

\begin{proof}
If $X$ is the zero vector field, then $X\lrcorner\vp=0$ at every point; in this case, we can take $X\lrcorner\vp=\d f$ where $f:M\to \R$ is the constant function $f(p)=0$ for all $p\in M$. Conversely, assume that $X$ is an arbitrary vector field such that $X\lrcorner\vp$ is exact. Then there exists some $1$-form $\al$ such that $X\lrcorner\vp=\d\al$. Using the $G_2$-metric defined by $\vp$, we have that
\begin{equation}
\begin{split}
6\langle X, X\rangle \Vol_{\vp}&=(X\lrcorner\vp)\w(X\lrcorner\vp)\w\vp \\
&=\d\al\w\d\al\w\vp=\d(\al\w\d\al\w\vp)\\
\end{split}
\end{equation}
From here, we find that, since $\pd M=\emptyset$, an application of Stokes' Theorem yields
\begin{equation}
\begin{split}
0&\leq 6||X||^2_{L^2}\Vol(M)=\int_M6\langle X, X\rangle \Vol_{\vp} \\
&=\int_M\d(\al\w\d\al\w\vp)=\int_{\pd M}\al\w\d\al\w\vp=0
\end{split}
\end{equation}
Since $\Vol(M)\neq 0$, we must have that $||X||_{L^2}=0$ proving that $X=0$ as desired.
\end{proof}

Corollary \ref{RocVecNonexistence} now follows immediately.

\begin{rem}
As pointed out by an anonymous referee, there is another instance of nonexistence, this time, for $G_2$ vector fields. Specifically, in the case of a compact torsion-free $G_2$-structure, i. e., a $G_2$-manifold, $G_2$ vector fields, being by definition Killing vector fields, will be parallel since $G_2$-manifolds are Ricci flat. In the case of a nontrivial $G_2$ vector field, there would necessarily be a reduction in the holonomy to a proper subgroup of $G_2$. Thus there are no nontrivial $G_2$ vector fields on a compact $G_2$-manifold with full $G_2$-holonomy.
\end{rem}

\begin{rem}
 The bracket constructed on Hamiltonian differential forms in Section \ref{HSMG} reduces to the usual Poisson bracket on functions in the case of a multisymplectic manifold of degree $2$, that is, a symplectic manifold, and it is a fundamental fact that diffeomorphisms which preserve the symplectic structure can be identified with the diffeomorphisms which preserve this Poisson bracket; however, no such characterization is available for diffeomorphisms which preserve the multisymplectic structure of a general multisymplectic manifold as can be seen using Corollary \ref{RocVecNonexistence}. Let $(M,\vp)$ be a closed manifold with closed $G_2$-structure, then the bracket operation operates either on a pair of Rochesterian $1$-forms, on a Rochesterian $1$-form and Rochesterian function or on a pair of Rochesterian functions; since the bracket on Rochesterian functions is necessarily zero and there are no nonzero Rochesterian $1$-forms, any smooth map $\Phi:M\to M$ then trivially preserves this bracket. 
\end{rem}

By the considerations in Section \ref{HSMG}, there are the following identifications
\begin{equation}
 \widetilde{\Om}^1_R(TM)=\Om^1_R(TM)\cong\widetilde{\Om}^1_R(T^*M)
\label{R11}
\end{equation}
\begin{equation}
 \widetilde{\Om}^2_R(TM)\cong\widetilde{\Om}^0_R(T^*M)\cong C^{\infty}(M)/\{f:M\to\R|f\text{ is locally constant}\}
\label{R20}
\end{equation}
\begin{equation}
 \widetilde{\Om}^1_{cR}(TM)=\Om^1_{cR}(TM)\cong\widetilde{\Om}^2_{cR}(T^*M)
\label{cR12}
\end{equation}
\begin{equation}
 \widetilde{\Om}^2_{cR}(TM)\cong\widetilde{\Om}^1_{cR}(T^*M)
\label{cR21}
\end{equation}
\begin{equation}
 \widetilde{\Om}^3_{cR}(TM)\cong\widetilde{\Om}^0_{cR}(T^*M)\cong C^{\infty}(M)/\{f:M\to\R|f\text{ is locally constant}\}
\label{cR30}
\end{equation}
\noindent
Equations \ref{R20}, \ref{cR30} immediately give the correspondence
\begin{equation}
 \widetilde{\Om}^2_R(TM)\cong\widetilde{\Om}^3_{cR}(TM)\cong C^{\infty}(M)/\{f:M\to\R|f\text{ is locally constant}\}
\end{equation}
\noindent
Equations \ref{R11}, \ref{cR21} and \ref{R1cR1} give an injective map
\begin{equation}
 \Om^1_R(TM)\hookrightarrow \Om^2_{cR}(TM)
\end{equation}
Explicitly, if $X\in\Om^1_R(TM)$, then by definition there exists a unique $\al\in\widetilde{\Om}^1_R(T^*M)$ such that
\begin{equation}
 X\lrcorner\vp=\d\al
\end{equation}
and, since $\widehat{\sv}_2$ is an isomorphism, there exists a unique $U\in\Om^2(TM)$ such that
\begin{equation}
 U\lrcorner\sv=\d\al.
\end{equation}
This completes the identifications of Theorem \ref{thmmain}. We remark that there is also the description of $\widetilde{\Om}^1_R(T^*M)$ as
\begin{equation}
 \widetilde{\Om}^1_R(T^*M)=\{\al\in\Om^1(T^*M):\d\al\in\Om^2_7(T^*M)\}/Z^1(M) 
\end{equation}

For our final consideration, we will need the following lemma which is a slight generalization of \cite[Lemma $2.4.6$]{Kari1} to general $2$-multivector fields, and the proof given here a direct adaptation of that proof.

\begin{lem}
 For $Q\in\Om^2(TM)$,
 \begin{equation}
  (Q\lrcorner\sv)\w(Q\lrcorner\vp)\w\sv=2|Q\lrcorner\vp|_{\vp}^2\Vol_{\vp}
  \label{QsvQvpsveq}
 \end{equation}
 where $|\cdot|_{\vp}$ is the norm induced by the $G_2$-metric $g_{\vp}$.
\label{QsvQvpsv}
\end{lem}

\begin{proof}
 By Equations \ref{Qflat4}, \ref{svpbe}
\begin{equation}
 Q\lrcorner\sv=\star(Q^{\flat}\w\vp)=2\be_7-\be_{14}
\end{equation}
where $\be_i=\pi_i(Q^{\flat})$ and $\pi_i$ is the projection $\pi_i:\Om^2(T^*M)\to\Om^2_7(T^*M)$. Writing $Q=\be_7+\be_{14}$ and using Equations \ref{Qflat3}, \ref{O27}, \ref{O214},
\begin{equation}
 \begin{split}
  (Q\lrcorner&\vp)\w\sv=\star(Q^{\flat}\w\sv)\w\sv=\star((\be_7+\be_{14})\w\sv)\w\sv \\
  &=\star(\be_7\w\sv)\w\sv+\star(\be_{14}\w\sv)\w\sv=\star\vp\w\star(\sv\w\be_7) \\
  &=3\star\be_7 \\
 \end{split}
\end{equation}
which implies
\begin{equation}
 |(Q\lrcorner\vp)\w\sv|^2_{\vp}=9|\be_7|^2
\end{equation}
By \cite[Equation $(2.8)$]{Kari1},
\begin{equation}
 |(Q\lrcorner\vp)\w\sv|_{\vp}^2=3|Q\lrcorner\vp|_{\vp}^2
\end{equation}
which together yield that
\begin{equation}
 |\be_7|_{\vp}^2=\frac{|Q\lrcorner\vp|^2}{3}
\end{equation}
Putting the pieces together yields
\begin{equation}
 \begin{split}
 (Q\lrcorner\sv)&\w(Q\lrcorner\vp)\w\sv=(2\be_7-\be_{14})\w(3\star\be_7) \\
 &=6\be_7\w\star\be_7+3\be_{14}\w\star\be_7=6|\be_7|_{\vp}^2\Vol_{\vp}+3g_{\vp}(\be_{14},\be_7)\Vol_{\vp} \\
 &=6|\be_7|^2_{\vp}\Vol_{\vp}=2|Q\lrcorner\vp|^2_{\vp}\Vol_{\vp} \\
 \end{split}
\end{equation}
\end{proof}

\begin{proof}[Proof of Proposition \ref{RcR2mvfdisjoint}]
 Let $Q\in\Om^2_R(TM)\cap\Om^2_{cR}(TM)$. Then there exist $f\in\Om^0_R(T^*M)$ and $\al\in\Om^1_{cR}(T^*M)$ such that
 \begin{equation}
  Q\lrcorner\vp=\d f
 \end{equation}
 and
 \begin{equation}
  Q\lrcorner\sv=\d\al
 \end{equation}
 Together with Lemma \ref{QsvQvpsv}, this implies that
 \begin{equation}
  2|Q\lrcorner\vp|_{\vp}^2\Vol_{\vp}=(Q\lrcorner\sv)\w(Q\lrcorner\vp)\w\sv=\d\al\w\d f\w\sv=\d(\al\w\d f\w\sv)
 \end{equation}
Using Stokes' Theorem, we find that
\begin{equation}
 \begin{split}
  0\leq&2||Q\lrcorner\vp||_{L^2}\Vol(M)=\int_M2|Q\lrcorner\vp|_{\vp}^2\Vol_{\vp} \\
  &=\int_M\d(\al\w\d f\w\sv)=\int_{\pd M}\al\w\d f\w\sv=0
 \end{split}
\end{equation}
Given that $\Vol(M)>0$, this implies that $Q\lrcorner\vp=0$.
\end{proof}

\begin{rem}
This is not necessarily the case for a general $G_2$-manifold. Again, consider $(\R^7,\vp_0)$ with $Q=\frac{\pd}{\pd x^6}\w\frac{\pd}{\pd x^7}$. Then $Q\in\Om^2_R(T\R^7)\cap\Om^2_{cR}(T\R^7)$ because
\begin{equation}
 Q\lrcorner\vp_0=\d x^1
\end{equation}
and
\begin{equation}
 Q\lrcorner\sv_0=\d(x^4\d x^5+x^2\d x^3)
\end{equation}
which also implies that $Q$ is nonzero in the quotient.
\end{rem}

\bibliographystyle{amsplain}
\bibliography{ham}
\end{document}